\documentclass{article} 
\usepackage{algorithm}
\usepackage{algpseudocode}
\usepackage{amsfonts}
\usepackage{url}
\usepackage{graphicx} 

\newcommand{\C}{\mathcal{C}}
\newcommand{\HH}{\mathcal{H}}
\newcommand{\minimals}{\mbox{minimals}}
\newcommand{\maximals}{\mbox{maximals}}
\newcommand{\msg}{\mbox{msg}}
\newcommand{\m}{\mathfrak{m}}
\newcommand{\F}{\mathfrak{F}}
\newcommand{\N}{\mathbb{N}}
\newcommand{\Z}{\mathbb{Z}}
\newcommand{\Np}{\mathbb{N}^p}

\newtheorem {Definition}{Definition}
\newtheorem {Proposition}{Proposition}
\newtheorem {proof}{Proof}
\newtheorem {Example}{Example}
\newtheorem {Lemma}{Lemma}
\newtheorem {Corollary}{Corollary}

\title{$\C$-semigroups with its induced order}
\author{Daniel Marín-Aragón and Raquel Tapia Ramos}
\date{}

\begin{document}

\maketitle

\begin{abstract}
Let $\C\subset\N^p$ be an integer polyhedral cone. An affine semigroup $S\subset\C$ is a $\C$-semigroup if $|\C\setminus S|<+\infty$. This structure has always been studied using a monomial order. The main issue is that the choice of these orders is arbitrary. In the present work we choose the order given by the semigroup itself, which is a more natural order. This allows us to generalise some of the definitions and results known from numerical semigroup theory to $C$-semigroups.
\end{abstract}

\section{Introduction}

Let $\N$ be the set of non-negative integers, we say that $S\subset\N$ is a numerical semigroup if it is an additive monoid and its complementary is finite. This structure has been study broadly in the literature (see for example \cite{libroRosales, ref1, ref4}). In \cite{libroRosales} it is proven that if $a_1,\ldots,a_e\in\N$ are coprimes, then $\langle a_1,\ldots,a_e\rangle=\{\lambda_1 a_1+\ldots+\lambda_e a_e\mid \lambda_1,\ldots \lambda_e\in\N\}$ is a numerical semigroup. The number of generator, i.e. $e$ is called the embedding dimension of $S$ and the number $|\N\setminus S|=g\in\N$, the genus of $S$. Other relevant invariants are: the Frobenius number defined as $F(S)=\max\{n\in\Z\mid n\notin S\}$, the conductor defined as $c(S)=F(S)+1$ and the left elements defined as $L(S)=|\{x\in S\mid x<F(S)\}|$. 

Related to these invariants there are still several open problems as Brass-Amorós's conjecture (see \cite{ref2}) or Wilf's conjecture (see \cite{wilf}) and, consequently, a lot of papers are published trying to solve them (see \cite{ref3, fromentin, wilfShalom,wilfpalestina}). The first conjecture, put forward 15 years ago, states that if $S_g$ is the set of all numerical semigroups with genus $g$ then $|S_{g+1}|\geq|S_g|$ for all $g\geq 0$. This conjecture is true for the numerical semigroups with genus less than $67$ (see \cite{fromentin}) and with genus greater than an unknown $g$ (see \cite{zhai}). On the other hand, the Wilf's conjecture was raised in 1978 and establishes that $e(S)L(S)\geq c(S)$.

In order to try to solve these conjectures, in \cite{nuestrowilf} the concept of numerical semigroup is generalised as follows: let $\C\subset\N^p$ be an integer polyhedral cone, then a $\C$-semigroup is an additive affine semigroup $S\subset\C$ with finite complementary set. Thus, the transfer the problem from $\N$ to $\Np$. In this new structure, the classical invariants of commutative monoids theory cited previously have been study (see for example \cite{Juande} or \cite{Adrian}).

In all these papers there is a huge issue, unlike in $\N$, in $\Np$ there is no a canonical total order. Therefore researchers have to choose a total order (as the graded lexicographical order) in order to define invariants as the Frobenius number but it is not always clear what happens when this order is changed. Our main goal is to show that this generalisation can be done in a different way which does not depend of the researcher's choice.

In this work, we use the order induced by the semigroup itself in order to do the generalisation, i.e., given $a,b\in S$ we say that $a\leq b$ if and only if $a-b\in S$. This order has already been applied to numerical semigroups (see for example \cite{ordeninducido}) but never to $\C$-semigroups. With this choice, the order is always clear and does not depend of an arbitrary choice. Thanks to this order, we open a new line for obtaining results which can be applied back to numerical semigroups. In order to provide the examples shown we have used a computer with a CPU Inter Core i7 8th Gen and the codes which are available at \cite{codigo}.

The content of this work is organised as follows: In Section \ref{S:preliminares}, we give the main definitions and some general results in order to provide background to the reader. We show our generalisation and how apply the changes to the main invariants. Then, in Section \ref{s:pesos} we introduce a new invariant, the quasi-elasticity (based on the concept of elasticity in numerical semigroups, see \cite{elasticity}) and show its properties. Section \ref{s:idemaxial} is devoted to idemaxial semigroups and we show bounds for computing some invariants in this family. Finally, in Section \ref{s:wilf}, we generalize the Wilf's conjecture by means of this partial order.

\section{Preliminaries}\label{S:preliminares}

Let $\C\subset\N^p$ be an integer polyhedral cone, with $\tau_1,\ldots,\tau_q$ its extremal rays and $h_1,\ldots,h_r$ its supporting hyperplanes. We say that $S$ is a $\C$-semigroup if it is an affine semigroup (i.e. $S$ is finitely generated, cancellative, reduced and torsion free), $S\subset\C$ and $|\C\setminus S|<+\infty$. The set $\HH(S)=\C\setminus S$ is called the set of gaps, or the gap set, of $S$.

We define the induced order of $S$ as
$$x\leq_{S}y\iff y-x\in S.$$
We use the symbol $\leq$ instead of $\leq_{S}$ if there is no risk of misunderstanding. Note that if $x\leq_{S}y$ then $x\leq_{\N^p}y$. The converse, trivially, is not true.

We recall that in a numerical semigroup, the multiplicity is the least element not zero of $S$ and the Frobenius number is the greatest element of $\Z$ which is not in $S$. We generalise these definitions as follows.  

\begin{Definition}
We define the set of multiplicities of $S$ as $\minimals_\leq(S)$ we denote this set as $\m(S)$.
We define the set of Frobenius as $\F(S)=\maximals_{\leq_{\C}}(\HH(S))$.    
\end{Definition}

As before, we use $\m$ and $\F$ when there is no risk of confusion.
Since $\HH(S)$ is finite and $\F\subset\HH(S)$ then $\F$ is also finite. Now we prove that $\m$ verifies this finiteness.

\begin{Proposition}
Let $S$ be a $\C$-semigroup, then the set $\m(S)$ is finite.
\end{Proposition}

\begin{proof}
For each supporting hyperplane $h_i\equiv a_1x_1+\ldots+a_px_p=0$ we define $h'_i(\alpha)\equiv a_1x_1+\ldots+a_px_p=\alpha$. We pick $\alpha_1,\ldots,\alpha_r$ such that if $f\in\F$ and $x\in h'_i(\alpha_i)$ then $fx\leq 0$.

We define $D=\{d\in\C\mid dx\leq 0,\ \forall x\in\cup_1^r h'_1(\alpha_i)\}$. Let $c\in\C\setminus 3D$, then $c\in \alpha\C$ for some $\alpha\in\N$, $\alpha>4$. Since there exists $d\in 3D\cap\C$ such that $c-d\in\C$, $\m\subset 3D$. Therefore $\m$ is finite.
\end{proof}

In \cite{Juande}, the authors give an algorithm for computing the gap set of a $\C$-semigroup, this allows us to give  Algorithm \ref{alg:computeM}
for computing $\m$.

\begin{algorithm}
\caption{Computing the set $\m(S)$}\label{alg:computeM}
\textbf{Input:} Set of gaps of $S$.\\
\textbf{Output:} $\m(S)$.
\begin{algorithmic}
\State Compute $\alpha_1,\ldots,\alpha_r$.
\State Define $2D$.
\State $X\gets\{First(D)\}$
\For{$d\in 2D$}
    \For{$d\in 2D$}
        \If{$d\leq x$}
            \State $X \gets X\setminus\{x\}\cup d$
        \EndIf
    \EndFor
\EndFor
\end{algorithmic}
\textbf{return} $X$.
\end{algorithm}

Given a $\C$-semigroup, it admits a unique minimal system of generators, denoted by $\msg(S)=\{a_1,\ldots,a_e\}$. That means that $S=\{n_1a_1+\ldots+n_ea_e\mid n_i\in\N,\ a_i\in \msg(S), i=1,\ldots,e\}$ and there is not a proper subset of $\msg(S)$ verifying this condition. The following result proves that the minimal system of generators is in fact the minimals of the $S$ with its induced order.

\begin{Proposition}
Let $S$ be a $\C$-semigroup, then $\msg(S)=\m(S)$.
\end{Proposition}

\begin{proof}
    $\msg(S)\subset\m(S)$: Let $x\in\msg(S)$ and we assume that $x\notin\m(S)$. Then there exists $y\in S$ such that $y\leq x$, i.e. $x-y=z\in S$. So $x=y+z$ which contradicts the fact that $x\in\msg(S)$.

    $\m(S)\subset\msg(S)$: Let $x\in\m(S)$ and we assume that $x\notin\msg(S)$. Then there exists $y,z\in S$ such that $x=y+z$, so $x-y=z\in S$ which contradicts the fact that $x\in\m(S)$.
\end{proof}

In \cite{Ojeda} the authors define the pseudo-Frobenius number of a $C$-semigroup as $a\in\HH(S)$ verifying $a+(S\setminus\{0\}\subset S$. The set of all the pseudo-Frobenius number of $S$ is denoted by $PF(S)$.

Our following result relates $PF(S)$ and $\F$.

\begin{Lemma}
If $S$ is a $\C$-semigroup then $\F\subset PF(S)$.
\end{Lemma}

\begin{proof}
    If $f\in\F$ and $f\notin PF(S)$ then $f+s=h\in\HH(S)$ for some $s\in S$. But that means $f-h\in S$ and this is a contradiction.
\end{proof}

Next example shows that the converse is not true.

\begin{Example}
    Let $\C$ be the cone spanned by $(1,0)$ and $(1,1)$ and $S=\C\setminus\{(1,1), (2,2)\}$. The element $(1,1)\in PF(S)$ but since $(2,2)-(1,1)\in\C$ we have that, $(1,1)\notin\F$.
\end{Example}

In \cite{Ojeda} the following definition for the Apery set is given.

\begin{Definition}
    Given $b\in S$, $Ap(S,b)=\{a\in S\mid a-b\in\HH(S)\}$.
\end{Definition}

Clearly, this set is finite ($|Ap(S,b)|\leq |\HH(S)|\leq +\infty$) and $Ap(S,b)-b\in\HH(S)$.
The following question arise: is there a relationship between $Ap(S,b)-b$ and $\F$?

\begin{Lemma}
    Let $S$ be a $\C$-semigroup, then $\F\subset Ap(S,b)-b$ for all $b\in S\setminus\{0\}$.
\end{Lemma}

\begin{proof}
    Let $f$ be in $\F$, then $f+b=s\in S$. Note that if $f+b=h\in\HH(S)$ then $h>f$ and this is a contradiction with $f\in\F$. Therefore, $s-b\in\HH(S)$ and $f\in Ap(S,b)-b$.
\end{proof}

In \cite{Ojeda}, the authors define the Frobenius elements, denoted by $F(S)$, as the gaps such that they are the maximum of the set of gaps for some term order of $\N^p$. We study the relation of this set with $\F$. Firstly, by \cite[pp. 72-73]{Cox}, every monomial ordering, $\preceq$, can be considered a weight order for some $a=(a_1,\ldots,a_d)\in\mathbb{R}^d_\geq$, i.e. $v\preceq w$ if and only if $v\cdot a\leq w\cdot a$ with $\cdot$ the inner product. Note that if these inner products are the same, we can choose another vector $\tilde{a}$ for tiebreaker. This is the same that saying that there exist a hyperplane such that divides the space in two region, one containing $v$ and the other one containing $w$. Next example shows that, in general, $\F\neq F(S)$.

\begin{Example}\label{e:fnotf}
    Let $\C$ be the cone spanned by $(1,0)$ and $(1,1)$ and $S=\C\setminus\{(1,0), (1,1), (2,0),(2,2),$ $(3,0), (3,1),(4,0)\}$. Then, $(3,0)\in\F$ but $(3,0)\not\in F(S)$.
\end{Example}

\begin{figure}
    \centering
\includegraphics[width=0.5\textwidth]{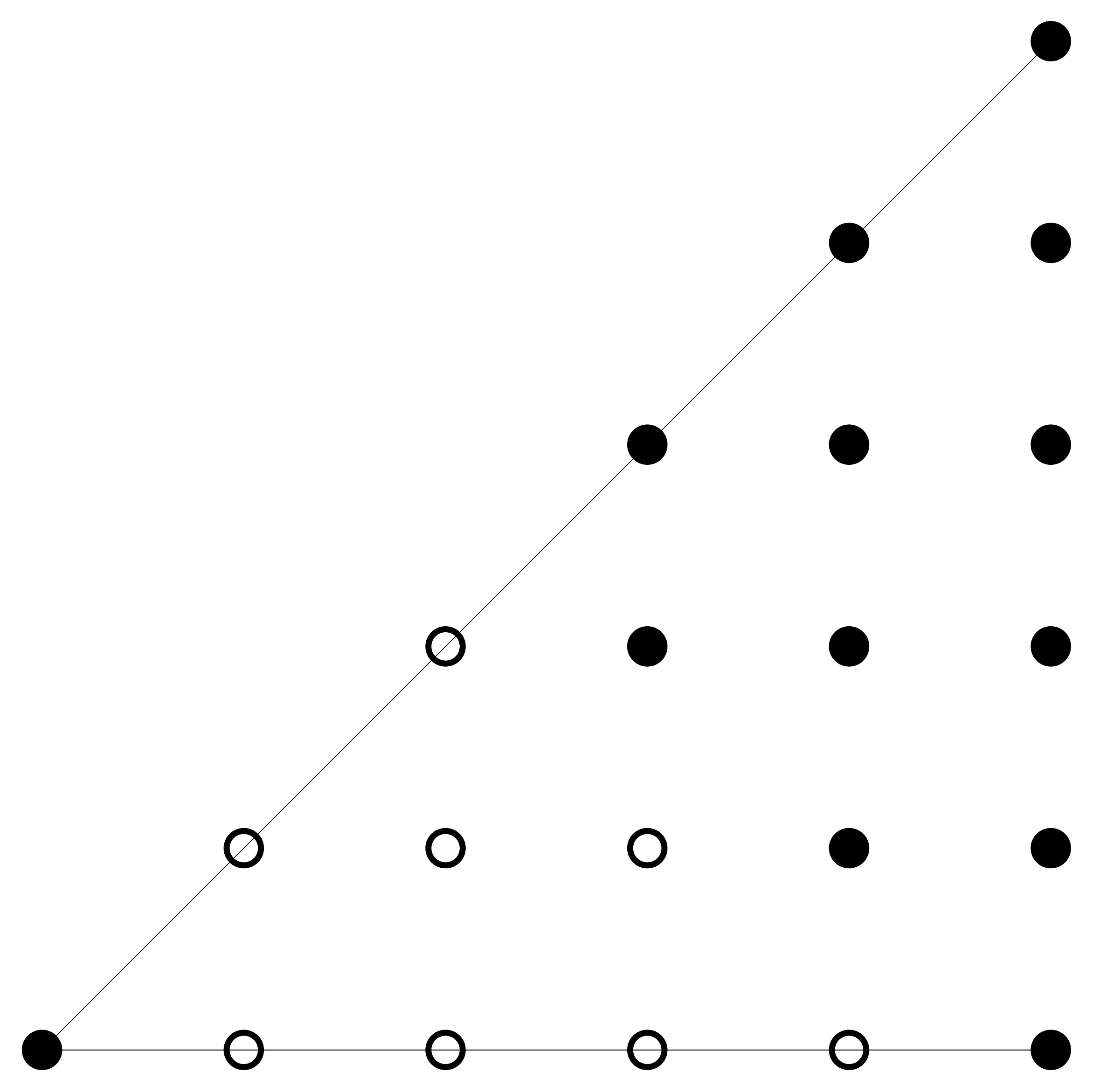}
    \caption{Figure of Example \ref{e:fnotf}.}
\end{figure}

The other inclusion, however, it is true.

\begin{Proposition}
    With the previous notation, $F(S)\subset\F$.
\end{Proposition}

\begin{proof}
    Let $f\in F(S)$ then there exists a hyperplane $\pi$ such that its normal vector has all its coordinates positives and that divides the space in two areas, $A_1$ and $A_2$ in such a way that $\HH(S)\cap A_1=\{f\}$. Therefore, $(f+\C)\cap\HH(S)=\{f\}$ and $f\in F$.
\end{proof}

The set $\F$ will be studied with more details in the following section.

\section{Weight sets}\label{s:pesos}

In \cite{Juande}, it is proven that if $\tau$ is an extremal ray of $\C$ and $S$ is a $\C$-semigroup, then $\tau\cap S$ is isomorphic to a numerical semigroup. However, the projection of the sum of the coordinates of the elements of $S$ does not verify this property as we show in this section.

This projection is not a capricious choice but is based on the one made in the study of factorization lengths in the case of numerical semigroups, see for example \cite{omega}.

\begin{Definition}
Given an element $(x_1,\ldots,x_p)\in\N^p$ we define its weight as $w(x_1,\ldots,x_p)=x_1+\ldots+x_p$. 
\end{Definition}

We can extend this definition to a set: if $A\subset\Np$ then $w(A)=\{w(a)\mid a\in A\}$.
Let $\Pi_t$ be the plane define as $x_1+\ldots+a_p=t$.
Given a $C$-semigroup, $S$, we associate  a set $W$ as follows:
$$x\in W \iff S\cap\Pi_x\neq\emptyset\ \&\ S\cap\Pi_x\cap\F=\emptyset.$$

Note that this set has the following properties:
\begin{itemize}
    \item It is unbounded.
    \item Its complementary is finite.
    \item It contains the zero element.
    \item In general, it is not closed by addition as the following example shows.
\end{itemize}

\begin{Example}
    Let $\C$ be the cone spanned by $(1,0)$ and $(1,1)$ and let $S=\C\setminus\{(1,1),(2,2)\}$. In this case, $\F(S)=\{(2,2)\}$ and $W=\N\setminus\{4\}$. 
\end{Example}

Therefore, $W$ is not a numerical semigroup. We are interested in study $w(\F)$. In particular, we are want to be able to compute the following invariant which is inpired by the elasticity studied in \cite{elasticity}.

\begin{Definition}
    Given $S$ a $\C$-semigroup, we define the quasi-elasticity of $w(\F)$ as $\rho(S)=\frac{\max(w(\F)}{\min(w(\F))}$.
\end{Definition}

Our first questions are: Given a fixed cone $\C$ can be found a $\C$-semigroup such that $\rho$ is as big as we want. If not, what value bounds it?

\begin{Proposition}
    Let $\C$ be a cone then $\rho$ is not bounded.
\end{Proposition}

\begin{proof}
    Let $S$ be a $\C$-semigroup with $\rho(S)=M\in\N$. Let $f_1,f_2\in\F$ such that $\omega(f_1)=\min(\omega(\F))$ and  $\omega(f_2)=\max(\omega(\F))$. Denoting by $\C_i=f_i+\C$, we have two cases:
    \begin{itemize}
        \item If $f_1\neq f_2$, then clearly $f_2\notin\C_1$. We choose $f_3\in C_2\setminus\C_1$. If we consider the $\C$-semigroup $(\C_1\setminus f_1)\cup(\C_3\setminus f_3)$, we obtain the result.
        \item If $f_1=f_2$, then we only have to choose $f_2$ and $f_3$ such that they are not comparable and $f_2\neq f_3$. Then we built the semigroup $(\C_2\setminus f_2)\cup(\C_3\setminus f_3)$ and apply the previous case.
    \end{itemize}
    Therefore, $\rho$ is unbounded.
\end{proof}

We have the following corollary.

\begin{Corollary}

Given a cone $\C$ we can find a sequence $S_k$, $k\in\N$ of $\C$-semigroups such that $\lim_{k\to\infty}\rho(S_k)=\infty$.
\end{Corollary}

\section{Idemaxial semigroups}\label{s:idemaxial}

In this section we introduce a new family of $\C$-semigroup, the idemaxial semigroups. As we recalled in the previous section,  if $S$ if a $\C$-semigroup with extremal rays $\tau_i$ with $1\leq i\leq k$ then $S\cap\tau_i$ is isomorphic to a numerical semigroup. We denote this numerical semigroup by $S_i$. In this section, we denote by $\phi_i$ the isomorphism such that $\phi_i(S\cap_tau_i)=S_i$.

\begin{Definition}
    We denote by $\pi_j$ the hyperplane containing the $j$-th elements of each $S\cap\tau_i$, by $F_i$ the Frobenius number of each $S_i$ and by $\pi_F$ the hyperplane containing $\phi_i^{-1}(F_i)$. We say that $S$ is an idemaxial semigroup if $S_1\approx\ldots\approx S_k$ and $S=(\cup_{i\geq 1}(\C\cap\pi_i))\cup\{x\in\C:x\cdot y > 0, y\in\pi_F\}$.
\end{Definition}

\begin{figure}
    \centering
\includegraphics[width=0.5\textwidth]{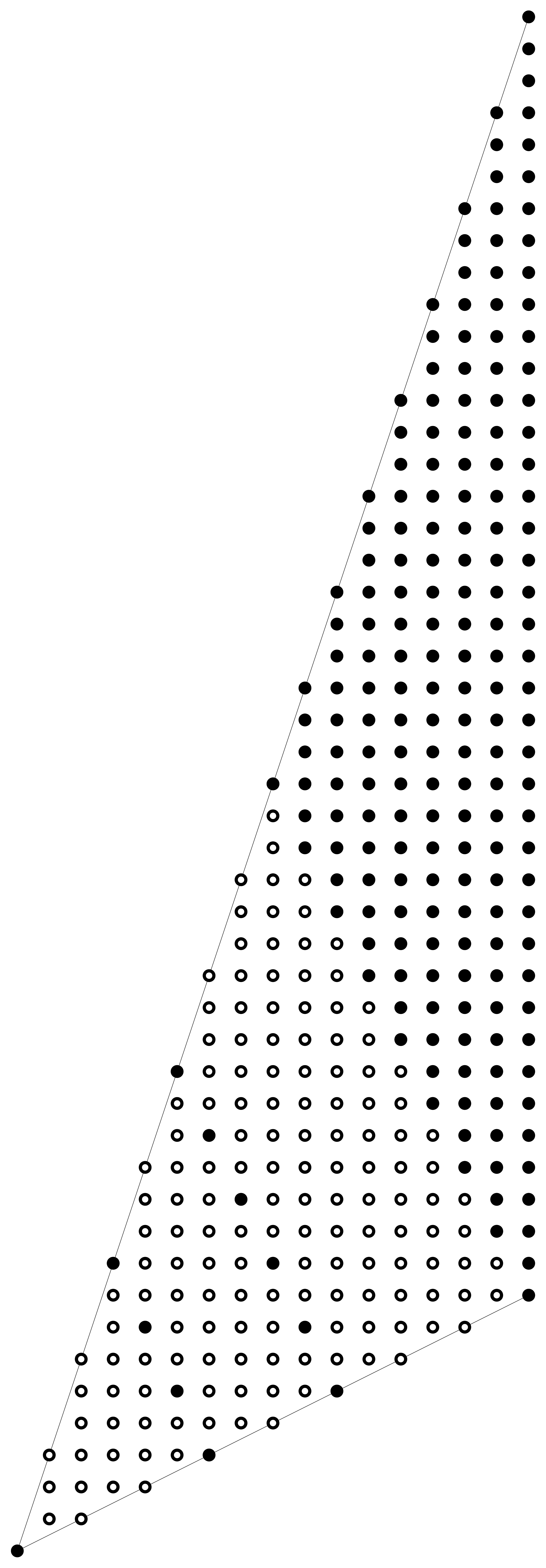}
    \caption{Example of an indemaxial semigroup with $S_1\approx S_2\approx\langle 3,5 \rangle$.}\label{e:idem}
\end{figure}

A graphical example of this kind of semigroup can be found in Figure \ref{e:idem}. In this case, $S_1$ and $S_2$ are isomorphic to $\langle 3,5 \rangle$. This family is usefull because it has good properties. We can, for example, find a bound where compute their Frobenius and pseudo-Frobenius set.

\begin{Proposition}
    Let $S$ be an idemaxial semigroups such that $\Phi(S\cap\tau_i) = S_i$ for some isomorphism $\Phi$. Let $m_i$ and $c_i$ the multiplicity and the conductor of $S_i$, respectively. Let $a_i$ be $\Phi^{-1}(m_i)$, $b_i$ be $\Phi^{-1}(c_i)$, $\pi_1$ the hyperplane which contains all $b_i-a_i$ and $\pi_2$ the hyperplane which contains all $b_i$. Then $\F\subset \{x\in\C\setminus S:x\geq_S \pi_1, x\leq_S\pi_2\}$.
\end{Proposition}

\begin{proof}
    Let $f\in\F$. Clearly, $f\leq_S\pi_2$. Since $f<\pi_2$, $\pi_1\in f+a_i\leq_S\pi_2\not\in S$ and this is a contradiction.
\end{proof}

We remind the following definition.

\begin{Proposition}
    Let $S$ be an idemaxial semigroup $PF_i$ the set of pseudo-Frobenius numbers of $S_i$ and $\pi_j$ the hyperplane containing the $j$-th pseudo-Frobenius number of each $S_i$. Then the set of pseudo-Frobenius numbers of $S$ is $\{\cup(\pi_j\cap\C)\}\subset PF$. 
\end{Proposition}

\begin{proof}
    By definition, $\{\cup(\pi_j\cap\C)\}\subset\HH(S)$. Moreover, if $s\in S$ and $f\in PF$, then $f+S\in\pi$ with $\pi$ containing elements of $S_i$ for all $i$, so $f+s\in S$.
\end{proof}

\section{Wilf's conjecture}\label{s:wilf}

We cannot end a work about semigroups without a brief mention of the Wilf's conjecture. Let $S$ be a numerical semigroup, Wilf's conjecture claims that $e(S)n(S)\geq c(S)$ where $e(S)$, $n(S)$ and $c(S)$ are its embedding dimension, the cardinal of its sporadic elements and its conductor (see \cite{wilf}). This conjecture has been generalized in several ways (see \cite{nuestrowilf, wilfgeneralizado}). In this works the authors generalise this conjecture for $\C$-semigroups using monomial total orders. On the other hand, in \cite{otrowilf} Wilf's conjecture is extended to generalized numerical semigroup using partial orders. In this section we are going to give even a more general conjecture using the induced order of a $\C$-semigroup.

Let $S$ be a $\C$-semigroup. We use the notation in \cite{otrowilf}:
$$c(S)=|\{a\in\C:a\leq b\mbox{ for some }b\in H(S)\}|,$$
$$n(S)=|\{a\in S:a\leq b\mbox{ for some }b\in H(S)\}|,$$
where $H(S)=\C\setminus S$.

Therefore we can pose the General Extended Wilf's conjecture:
$$e(S)n(S)\geq pc(S).$$

This conjecture has been checked in a computational way and no counterexamples have been found.

\section{Discussion}

The traditional approach to study $C$-semigroups was to set a total order in a completely arbitrary way and then study the desired properties. This approach has the following issue: if the order is changed, the results obtained may not be true.

In this paper we replace the total order chosen by the researchers by the partial order induced by the semigroup itself. As this order depends exclusively on the semigroup, it does not depend on the researcher, so the results obtained are less artificial.  For example, when classical invariants were generalised in the study of semigroups, such as the Frobenius number or multiplicity, being completely dependent on the order, they changed as the order changed.

Moreover, since numerical semigroups have only one Frobenius element and only one multiplicity, $C$-semigroups were forced to have only one of these elements. However, with our method of study, we have a set of elements for each of these invariants, which is more natural since we are in a higher dimension. In addition to this, in this work, we have given a family of $C$-semigroups with good properties.

This work aims to lay the foundations for future work on $C$-semigroups.

\end{document}